\documentclass[12pt]{amsart}
\usepackage{epsfig,color}

\makeatother

\theoremstyle{definition}

\def\fnum{equation} 
\newtheorem{Thm}[\fnum]{Theorem}
\newtheorem{Cor}[\fnum]{Corollary}

\newtheorem{Lem}[\fnum]{Lemma}

\newtheorem{Pro}[\fnum]{Proposition}

\numberwithin{equation}{section}

\newcommand{\Entropy}{{\text{S}}}

\newcommand{\osc}{{\text{osc}}}

\newcommand{\dist}{{\text {dist}}}

\newcommand{\tv}{{\tilde{v}}}

\def\ZZ{{\bold Z}}
\def\RR{{\bold R}}

\newcommand{\e}{{\text {e}}}

\newcommand{\Var}{{\text {Var}}}

\newcommand{\cE}{{\mathcal{E}}}

\newcommand{\cA}{{\mathcal{A}}}
\newcommand{\cB}{{\mathcal{B}}}

\newcommand{\cH}{{\mathcal{H}}}

\newcommand{\eqr}[1]{(\ref{#1})}

\title[Committee ranking]{Committee ranking}
\author[Colding]{Tobias Holck Colding}%
\address{MIT, Dept. of Math.\\
77 Massachusetts Avenue, Cambridge, MA 02139-4307.}
\email{colding@math.mit.edu and minicozz@math.mit.edu}
\author[Minicozzi]{William P. Minicozzi II}%
 
\thanks{The  
authors
were partially supported by NSF Grants DMS 1404540 and DMS 1206827.}

\begin{document}

\maketitle

\begin{abstract}
This paper deals with interactions between committee members as they rank a large list of applicants for a given position and eventually reach  consensus.  We will see that for a natural deterministic model the ranking can be described by solutions of a discrete quasilinear heat equation with time dependent coefficients on a graph.

We show first that over time consensus emerges exponentially fast.  Second,  if  there are clusters of members whose views are closer than those of the rest of the committee, then over time the clusters' views become closer at a faster exponential rate than the views of the entire committee.  

 We will also show that the variance of the rankings decays a definite amount, independent of the initial variance, when the influence of the members  does not decay too quickly as opinions differ.  When the influence between members is exactly a negative power, then the variance is convex and satisfies a  three circles theorem.
 \end{abstract}

\section{Introduction}

As a committee ranks candidates, members adjust their rankings taking into account  opinions of   others.
   If two members rank applicants close, then they are each more likely to reorder their rankings to move them closer together.  If their views  
    are far apart, then they are less likely to make significant changes.  They continue this process of ranking, deliberating, reranking and deliberating again.

 We will consider a deterministic model  that leads to a discrete quasilinear heat equation with time dependent coefficients on a graph; see \eqr{e:oneptwo}.  We will analyze solutions and begin by proving two general features.  First,  solutions converge exponentially to an equilibrium; see Corollary \ref{c:main1}.  Second, clusters of closely aligned members converge at a faster rate than the committee as a whole; see Theorem \ref{l:cluster1}.
 
A  consequence of our results is that if most of a committee has relatively close views but there are a few outliers, then the bulk of the committee will reach consensus much faster than the entire committee.  The consensus reached among the bulk will be influenced by the outliers, even though it may still be quite far from their views.  The further the outliers are from the bulk, the less influence they will have before the bulk reaches a consensus.

We also show that if the influence of members on each other does not decay too quickly as opinions differ, then the variance of opinions decays a definite amount  independent of the initial variance.  As a consequence, we get an upper bound for the time to consensus; see Corollary \ref{c:vardecay}.  This upper bound depends only on the initial variance.  This bound for the decay is stronger than the  exponential decay above.  This is because the exponential rate is very slow when there are outliers.

When the influence between members is exactly equal to a negative power in the difference of opinion, then the variance is convex and 
satisfies a nonlinear three circles theorem; see Theorem \ref{t:threecircles}.   This holds independently of the number of candidates and committee members.

Our results work equally well whether  the influence of one member on another depends on how closely aligned  their rankings are of a particular candidate, or of the entire set of the candidates.
For the exponential decay, our results work for ranking maps that take values in a Banach space, finite or infinite dimensional.  For the decay of the variance,   the ranking map takes values in a Hilbert space.

We also relate our parabolic equation to the hyperbolic equation known from the classical $n$-body problem in celestial mechanics for predicting the individual motions of a group of celestial objects interacting with each other gravitationally (see Section \ref{s:nbody}).  
In the $n$-body problem, the force between objects decays quadratically in the distance between them.

\vskip1mm
This article grew out of a question of  Sreeram Kannan, Sanjeev Khanna and Madhu Sudan; we are grateful to them for discussions, \cite{KKS}. 
 
\section{The equation}

Let $\rho:[0,\infty)\to [0,1]$  be a   nonincreasing function. 
Suppose that the committee has $d+1$ members and they all meet together and discuss their rankings of $n$ candidates.  Let $\Gamma$ be the complete graph with $d+1$ vertices so all vertices $v$ are adjacent to each other.  We will write $\tilde v\approx v$ for $\tilde v$ adjacent to $v$.  The ranking map $f:\Gamma\times \ZZ_+\to \RR^n$ is given recursively by
\begin{align}	\label{e:firstrho}
f(v,t+1)=\frac{1}{d}\sum_{\tilde v\approx v}\left[f(v ,t)\, (1-\rho (|f(v,t)-f(\tilde v,t)|))+f(\tilde v,t)\, \rho (|f(v,t)-f(\tilde v,t)|)  \right]\, .
\end{align}
Each member's new  ranking is a weighted average of their old ranking and the rankings of the other members.
The  $\rho$ is the relative weight given to the other members' ranking.  Since $\rho(s)$ is nonincreasing in $s$,
they give less weight to rankings that are very different from their own.   We will take $\rho \in [0,1]$, but one could restrict to $\rho \leq 1/2$ since it would be natural that they put more weight on their own opinion.

Obviously equation \eqr{e:firstrho}
 may be rewritten as   the nonlinear equation for $f$
\begin{align}	\label{e:oneptwo}
\partial_tf&= \frac{1}{d} \sum_{ \tilde v\approx v}   [f(\tilde v,t)-f(v,t)]\,\rho (|f(\tilde v,t)-f(v,t)|)  \, ,
\end{align}
where $\partial_t\,f=f(v,t+1)-f(v,t)$.

This allows for very general models.  The following simple examples are illustrations.
\subsubsection{$\rho$ constant} 
In the extreme case where $\rho \equiv  0$,  the committee members views are rigid and their ranking remains unchanged. 
More generally, when $\rho=p$ for some constant $0<p<1$, we get the ordinary discrete heat equation (cf. \cite{C}, \cite{G}, \cite{S}, \cite{Su}), up to a constant,
\begin{align}
\partial_tf(v,t) =\frac{p}{d}\,\sum_{\tv \ne v}(f(\tv ,t)-f (v,t)) = p \, \Delta f (v,t)  \, . 
\end{align}
 Therefore the maps $f (v,t)\to g$ exponentially fast as $t\to\infty$.   In other words,
  the ranking of the committee becomes synchronized exponentially fast.    This is also a special case of the next example.  
\subsubsection{$\rho$ bounded from below} 
Even in the more general case where $0< p  \leq \rho $ for some positive constant $p$ we claim that as $t\to \infty$ the maps $f(v,t)\to g$ exponentially fast so that again the ranking of the committee becomes synchronized exponentially fast.      This will be a consequence of Theorem \ref{t:main1}.  
\subsubsection{Clusters with no interaction}
Suppose that the committee can be divided into clusters $A_1,\cdots, A_{\alpha}$ where 
$
|f(v, 0)-f(\tilde v,0)|< R
$
if $v$ and $\tilde v$ belong to the same cluster and
$
|f(v,0)-f(\tilde v,0)|\geq  R
$
otherwise.  If $\rho (R)=0$, then the views of one cluster are not influenced by another.

\section{Time dependent nonlinear elliptic operators on graphs}
Let $\mu$ be a positive function on the oriented edges of the graph $\Gamma$ and define the discrete linear elliptic operator\footnote{This operator is close to what is typically called a weighted Laplacian, but differs since we do not require that $\frac{1}{d}\sum_{\tilde v\approx v} \mu^t_{v,\tilde v}=1$ or even just constant in $v$.  In our applications $\mu$ will be allowed to be time dependent.} on functions on $\Gamma$ by
\begin{align}	\label{e:3p1}
L_{\mu}\,u (v) =\frac{1}{d}\sum_{\tilde v\approx v} [u(\tilde v)-u(v)]\,\mu_{v,\tilde v}\, .
\end{align} 
We will assume that $1\geq \mu_{v,\tilde v} \geq 0$.  This operator is elliptic exactly when $\mu>0$.  Recall that the graph Laplacian is where $\mu_{v,\tilde v} \equiv 1$.
In our first application  
\begin{align}
\mu_{v,\tilde v}= \rho (|u(v)-u(\tilde v)|)\, .  
\end{align}
In this case $\mu$ does not depend on the orientation of the edge; however,   the operator $L_{\mu}$ is nonlinear.  We will let $\cE$ denote the set of unoriented edges
\begin{align}	\label{e:edges}
	\cE = \{ (v_1 , v_2) \, | \, v_1 \ne v_2 \in \Gamma \} \mod (v_1 , v_2) \approx (v_2 , v_1) \, .
\end{align}

Next we will also allow $u$ to depend on time so that $\mu=\mu^t_{v,\tilde v}$ also depends on time and consider the discrete quasilinear time dependent heat equation
\begin{align}
\partial_tu=L_{\mu}\,u\, .
\end{align}

\begin{Lem}    \label{l:preserve}
Suppose $u:\Gamma\to \RR$.  If $\mu$ is independent of the orientation of the edges\footnote{This means that $\mu_{v,\tilde v}=\mu_{\tilde v,v}$.}, then 
\begin{align}
\sum_{v}\left(L_{\mu}\,u\right)\,(v)=0\, .
\end{align}
\end{Lem}

\begin{proof}
A straightforward calculation gives
\begin{align}
\sum_{v}\left(L_{\mu}\,u\right)(v) &=\sum_{v}\frac{1}{d}\sum_{\tilde v\approx v} [u(\tilde v)-u(v)]\,\mu_{v,\tilde v}\notag\\&=\frac{1}{d}\sum_{(v_1 , v_2) \in \cE} \left([u(v_2)-u(v_1)]\,\mu_{v_1,v_2}+[u(v_1)-u(v_2)]\,\mu_{v_2,v_1}\right)=0\, ,
\end{align}
where the last equality follows since $\mu$ does not depend on the orientation of an edge.
\end{proof}

As a corollary, when $\mu$ is independent of the orientation of the edges, the overall view of the whole committee remains unchanged over time even if the view of individual members may change as they become more aligned with the views of other members; this is:

\begin{Cor}	\label{c:intupre}
Suppose that $\partial_t\,u=L_{\mu}\,u$ and $\mu$ is independent of the orientation of the edges, then  
$\sum_{v}u(v,t)
$
is constant in $t$.  
\end{Cor}

We will let $\cA_u$ denote the average of a function $u$ on $\Gamma$
\begin{align}
	\cA_u = \frac{1}{d+1} \, \sum_{v \in \Gamma} u(v) \, .
\end{align}
By Corollary \ref{c:intupre}, $\cA_u$ is constant in time when 
$\partial_t\,u=L_{\mu}\,u$ and $\mu$ is independent of the orientation of the edges.

Using the parabolic maximum principle,
 we have that the most extreme views   moderate over time:

\begin{Pro}  \label{p:maxprin}
If $\partial_t\,u=L_{\mu}\,u$ on a graph $\Gamma$, then (in time)
\begin{align}
M_t=\max_{v}u (v,t)\downarrow {\text{ and }}
m_t=\min_{v} u (v,t)\uparrow\, .
\end{align}
\end{Pro}

\begin{proof} 
A simple computation shows
\begin{align}	\label{e:prma}
u(v,t+1)&=u(v,t)+\left(L_{\mu}\,u\right)(v,t)
= \left(1-\frac{1}{d}\sum_{\tilde v\approx v}\mu_{v,\tilde v}\right)\,u(v,t)+\frac{1}{d}\sum_{\tilde v\approx v}\mu_{v,\tilde v}\,u(\tilde v,t)\\
&\leq \left(1-\frac{1}{d}\sum_{\tilde v\approx v}\mu_{v,\tilde v}\right)\,\max_{w}u(w,t)
+\frac{1}{d}\sum_{\tilde v\approx v}\mu_{v,\tilde v}\,\max_{w}u(w,t)\notag\\
&= \max_{w}u(w,t)\, .\notag
\end{align}
This proves that $M_{t+1} \leq M_t$.  The monotonicity of the minimum follows similarly.
\end{proof}

Recall that a vector field on the $d$-regular graph $\Gamma$ is a map from $\Gamma$ to $\RR^d$.   The gradient of a function $u$ is the vector field $\nabla^{\tilde v} u(v)=u(\tilde v)-u(v)$, where $\tilde v\approx v$.    For a vector field $V$ on the graph we will set
\begin{align}
\|V\|_{\infty}=\max_{\tilde v\approx v} |V^{\tilde v}(v)|\, .
\end{align}

Loosely speaking when $\mu_{v,\tilde v}=\rho (|u(v)-v(\tilde v)|)$ our discrete heat equation will be a graph version of a quasilinear heat equation on functions on $\RR^n\times\RR$ of the form 
\begin{align}
\partial_t\,u=\sum_{i,j}a_{i,j}(\nabla u)\,\frac{\partial^2u}{\partial x_i\,\partial x_j}\, ,
\end{align}
where $(a_{i,j})_{i,j}$ is elliptic.  

This suggests that the parabolic maximum principle should give a gradient estimate.  We will show this next   ($\mu$ is allowed to depend on the orientation of the edges).   We will see a strengthening of this in the next section.

\begin{Cor}  \label{c:gradbound}
If $\partial_t\,u=L_{\mu}\,u$ on the graph $\Gamma$, then (in time)
\begin{align}
\|\nabla u \|_{\infty}\downarrow\, .
\end{align}
\end{Cor}

\begin{proof}
By Proposition \ref{p:maxprin}, for all $v_1$, $v_2$
\begin{align}
|u(v_1,t+1)-u(v_2,t+1)|\leq M_{t+1}-m_{t+1}\leq M_t-m_t\, .
\end{align}
The claim follows from this.
\end{proof}

Since every pair of distinct vertices in $\Gamma$ is connected by an edge, the norm of the gradient is equal to the oscillation of the function
$\osc_u=\max_{v,w} |u(v)-u(w)|=\|\nabla u \|_{\infty}\, .$
 The obvious interpretation of Corollary \ref{c:gradbound}
 is that the difference in the extreme views of the committee narrows   over time.

\section{Convergence to consensus via decay of the gradient}

It is natural to consider models where the influence of one member on another depends on how closely aligned  their rankings are of all of the candidates.
In this case, we get a nonlinear parabolic system of equations for
the entire vector-valued ranking map.

\subsection{Vector-valued maps}
 
Suppose  that 
$
	u: \Gamma \to \cB
$
 where $\cB$ is a Banach space\footnote{If  there are finitely many candidates, then the number of them is the dimension of the Banach space $\cB$.}.  The Banach norm will be denoted by $\| \cdot \|_{\cB}$. The $\|\cdot\|_{\infty}$ norm of $\nabla u$ is given by
\begin{align}
	\| \nabla u   \|_{\infty} = \max_{v , \tv} \, \|u(v ) - u(\tv )\|_{\cB} \, .
\end{align}
Since $\cB$ is a vector space, we can define the operator $L_{\mu}$ as in \eqr{e:3p1} with coefficients $\mu_{v,\tv} $ in $[0,1]$.  An important special case is when $\mu_{v,\tv} = \rho (\| u(\tv) - u(v) \|_{\cB})$ where $\rho:\RR \to [0,1]$ is nonincreasing.  

The time one map $A_u$ is given by
	$A_u(v)=u(v)+\left(L_{\mu}\,u\right)(v)$, so that
 $\partial_t\,u=L_{\mu}\,u$ precisely when $u(v,t+1)=A_{u(\cdot,t)}(v)$.  
Arguing as in the case where $\cB = \RR$ (see Lemma \ref{l:preserve}),  
the time one map preserves the average if the $\mu$'s are symmetric.
 
	 A simple computation shows
\begin{align}	\label{e:simplecompu}
A_u(v)& = \left(1-\frac{1}{d}\sum_{\tilde v\approx v}\mu_{v,\tilde v}\right)\,u(v)+\frac{1}{d}\sum_{\tilde v\approx v}\mu_{v,\tilde v}\,u(\tilde v)\, .
\end{align}
Thus, the time one map is a convex combination of the rankings at the previous time and  the triangle inequality gives that
\begin{align}	\label{e:maxdown}
	\max_v \| A_u (v) \|_{\cB} \leq \max_v \| u (v) \|_{\cB} \, .
\end{align}

  \subsection{Exponential decay}

The next theorem gives that the gradient is not only nonincreasing, generalizing Corollary \ref{c:gradbound} to vector-valued maps, but the gradient decays when $d \geq 2$.
   This restriction is necessary as the connected graph with  two vertices is bipartite and even the linear heat equation on this graph has solutions that oscillate 
without any decay.{\footnote{This oscillation can happen when $\mu = 1$, but not when $\mu < 1$.}}  The exponential decay does not require symmetry of the $\mu$'s, it holds even when the dimension is infinite, and it is uniform as $d \to \infty$.

 \begin{Thm}     \label{l:iterateStepV}
If $1\geq \mu\geq a\geq 0$, then  
\begin{align}
	\| \nabla A_u  \|_{\infty} \leq  \e^{ - a \, \left( \frac{d-1}{d} \right) } \, 
	\| \nabla u   \|_{\infty}  \, .
\end{align}
\end{Thm}

\vskip1mm
When $a>0$ is small, then  $\e^{ - a \, \left( \frac{d-1}{d} \right) } $ is approximately $1 - a \, \left(\frac{d-1}{d} \right)$.
The exponential decay will rely on two  lemmas for Banach spaces and a corollary of them.  The point is that \eqr{e:simplecompu} gives the values of the time one map as convex combinations of the initial values; hence, the next lemma shows that the oscillation is nonincreasing.

\begin{Lem}	\label{l:cvxcom}
Suppose that $x_1 , \dots , x_{d+1}$ are vectors in a Banach space $\cB$ with norm $\| \cdot \|_{\cB}$.  If 
$y= \sum_i a_i x_i$ and $z = \sum_i b_i x_i$ where $a_i , b_i \geq 0$ and $\sum_i a_i = \sum_i b_i = 1$, then
\begin{align}
	\| y - z \|_{\cB} \leq \max_{i,j} \,\| x_i - x_j \|_{\cB} \, .
\end{align}
\end{Lem}

\begin{proof}
Set $c_{ij} =a_i b_j$, so that $c_{ij} \geq 0$, $\sum_i c_{ij} = b_j$, $\sum_j c_{ij} = a_i$, and $\sum_{i,j} c_{ij} =1$.
We have 
\begin{align}
	\| y - z \|_{\cB} &= \| \sum_{i,j} c_{ij} (x_i - x_j) \|_{\cB} \leq   \sum_{i,j} c_{ij} \|(x_i - x_j) \|_{\cB} 
	\leq \left( \sum_{i,j} c_{ij} \right) \,  \max_{i,j} \,\| x_i - x_j \|_{\cB}  \notag \\
	&=  \max_{i,j} \,\| x_i - x_j \|_{\cB}  \, .
\end{align}
\end{proof}

 \begin{Lem}	\label{l:induct}
Let $ x_i, y  \in \cB$ be as in Lemma \ref{l:cvxcom}.   If $a_j  \geq c > 0$ for some $j$, then $y$ is in the convex hull of 
vectors $\bar{x}_i$ where 
\begin{align}
	\bar{x}_j = x_j {\text{ and }} \bar{x}_i = x_j + (1-c) (x_i - x_j) {\text{ for }} i \ne j \, .
\end{align}
In fact, $y = \sum_i \bar{a}_i \, \bar{x}_i$ where $\bar{a}_i \geq 0$ with $\sum_i \bar{a}_i =1$ are given by
\begin{align}	
	\bar{a}_j = \frac{a_j -c }{1-c}   {\text{ and }} \bar{a}_i = \frac{a_i}{1-c}  {\text{ for }} i \ne j \, .
\end{align}
\end{Lem}

\begin{proof}
This follows since
\begin{align}
	(1-c) \, \sum_i a_i x_i =(a_j -c) \, x_j + \sum_{i \ne j} a_i \, \left( x_j + (1-c) (x_i - x_j) \right) \, .
\end{align}

\end{proof}

\begin{Cor}	\label{c:contraction}
Let $ x_i, y , z \in \cB$ be as in Lemma \ref{l:cvxcom}.   If $a_j  , b_j \geq c > 0$ for  at least ${d}_0$ of the $j$'s, then
\begin{align}
	\| y - z \|_{\cB} \leq (1-c)^{ {d_0} } \, \max_{i,j} \| x_i - x_j \|_{\cB} \, .
\end{align}
\end{Cor}

\begin{proof}
This will follow by applying Lemma \ref{l:induct} for each of the $d_0$ indices $j$ where $a_j , b_j \geq c$.  Each time we apply the lemma, the convex hull containing both $y$ and $z$ is dilated down by a factor of $(1-c)$.
The initial diameter is bounded by $\max_{i,j} \| x_i - x_j \|_{\cB}$
by Lemma \ref{l:cvxcom}.
\end{proof}
 
  \begin{proof}[Proof of Theorem \ref{l:iterateStepV}]

Given $v_1 \ne v_2$,  the vectors $A_u(v_1) $ and $ A_u(v_2)$ are both given by \eqr{e:simplecompu} as  convex combinations of the $u(v)$'s 
\begin{align}	 \label{e:thecoeffs}
A_u(v_i) & = \left(1-\frac{1}{d}\sum_{\tilde v\approx v_i}\mu_{v_i,\tilde v}\right)\,u(v_i)+\frac{1}{d}\sum_{\tilde v\approx v_i}\mu_{v_i,\tilde v}\,u(\tilde v) \, . 
\end{align}
Thus, 
if we set $x_i = u(v_i ) \in \RR^n$ for $i=1 , \dots , d+1$, then $y \equiv A_u(v_1 )$ and $z \equiv A_u(v_2)$ both lie in the convex hull of the $x_i$'s since the coefficients in \eqr{e:thecoeffs} are nonnegative and add up to one.
Consequently, 
  Lemma \ref{l:cvxcom}  gives that
\begin{align}
	\| A_u(v_1) - A_u(v_2) \|_{\cB} \leq \max_{v , \tv} \, \|u(v ) - u(\tv )\|_{\cB} = \| \nabla u (\cdot ) \|_{\infty} \, .
\end{align}
Since this holds for every $v_1$ and $v_2$, this gives the theorem when $a=0$.  

Suppose now that $a> 0$.  It follows from \eqr{e:thecoeffs} that at least $d-1$ of the $j$'s (in fact every $j \ne 1,2$) are at least $\frac{a}{d}$ for both $y$ and $z$.  Therefore, we can apply Corollary \ref{c:contraction} to get that
\begin{align}
	\| A_u(v_1) - A_u(v_2) \|_{\cB} \leq \left( 1 - \frac{a}{d} \right)^{d-1} \, \| \nabla u (\cdot ) \|_{\infty} \, .
\end{align}
To complete the proof, use that $\log (1+x) \leq x$ to get
\begin{align}
	\left( 1 - \frac{a}{d} \right)^{d-1}= \e^{ (d-1) \, \log \left( 1 - \frac{a}{d} \right) } \leq
		 \e^{ -a \, \left( \frac{d-1}{d} \right) } \, .
\end{align}

\end{proof}

  Iterating Theorem \ref{l:iterateStepV} gives 
exponentially
 fast convergence towards consensus (like Theorem \ref{l:iterateStepV} this does not require that $\mu$ is symmetric):

\begin{Thm}   \label{t:main1}
If $\partial_t\,u=L_{\mu}\,u$ and $1\geq \mu\geq a>0$, then 
\begin{align}	\label{e:412}
	\| \nabla u(\cdot,t)  \|_{\infty} &\leq   \e^{ -a \, t\, \left( \frac{d-1}{d} \right) }  \, 
	\| \nabla u (\cdot,0)  \|_{\infty}  \, , \\
     \label{e:413}
\left\|u(\cdot,t)- \cA_{u(\cdot , t)} \right\|_{\infty}
&\leq  \frac{2d}{d+1} \, \e^{ -a \, t\,  \left( \frac{d-1}{d} \right) } \,\left\|u(\cdot,0)- \cA_{u (\cdot,0) } \right\|_{\infty} \, .
\end{align}
\end{Thm}

\begin{proof}
The first inequality follows immediately from iterating Theorem \ref{l:iterateStepV}.  
The second follows from the first since $
	\|\nabla u (\cdot , 0) \|_{\infty}\leq 2\,\left\|u(\cdot,0)- \cA_{u (\cdot,0) } \right\|_{\infty}$ and
\begin{align}
 \left\|u(w,t)-\cA_{u (\cdot , t)} \right\|_{\cB} =\left\| \frac{1}{d+1}
	\sum_v[u(w,t)-u(v,t)]\right\|_{\cB}  \leq \frac{d}{d+1} \, \|\nabla u(\cdot,t)\|_{\infty} \, .
\end{align}
\end{proof}

\begin{Cor}   \label{c:main1}
If $\partial_t\,u=L_{\mu}\,u$, $\mu^t_{v,\tilde v}=\rho (\|u(v,t)-u(\tilde v,t)\|_{\cB})$, and $a=\rho (\|\nabla u(\cdot,0)\|_{\infty})>0$, then  \eqr{e:412} and \eqr{e:413} hold.
\end{Cor}

\begin{proof}
This follows from Theorem \ref{t:main1} since $\|\nabla u(\cdot,t)\|_{\infty}\leq \|\nabla u(\cdot,0)\|_{\infty}$ and $\rho$ is monotone so the lower bound on the $\mu$'s is preserved.  
\end{proof}

\section{$n$-body problem}  \label{s:nbody}
Even though our equation is parabolic there are formal similarities between it and the (hyperbolic) equation that describes the classical $n$-body problem.

In physics, the $n$-body problem, \cite{Mi}, \cite{Mo1}, \cite{Mo2},
 is the problem of predicting the individual motions of a group of celestial objects interacting with each other gravitationally.  Solving this problem has been motivated by the desire to understand the motions of the Sun, Moon, planets and the visible stars. In the 20th century, understanding the dynamics of globular cluster star systems became an important $n$-body problem.

The $n$-body problem is described by the system of differential equations\footnote{This equation is derived by combining Newton's second law: $F=m\, a$ (force equal to mass times acceleration) with Newton's law of gravity that the force between two bodies with masses $m_1$ and $m_2$ is inverse proportional to the square of the distance between them: $F=G\frac{m_1\,m_2}{r^2}$.}:
\begin{align}
\frac{d^2\mathbf{x}_i(t)}{dt^2}=G \sum_{k=1,k\neq i}^n \frac{m_k \left(\mathbf{x}_k(t)-\mathbf{x}_i(t)\right)}{\left|\mathbf{x}_k(t)-\mathbf{x}_i(t)\right|^{3}}\, ,
\end{align}
where $\mathbf{x}_k(t)$ denotes the position of the $k$th object at time $t$, with mass $m_k$, and $G$ is the gravitational constant.  

This can also be thought of as a nonlinear wave equation on a graph (or rather on a graph $\times \RR$).  Namely, let $\Gamma$ be the complete graph with $M=\sum_{k}m_k$ vertices and with $n$ clusters where the $k$th cluster consists of $m_k$ vertices.  If we set $\rho (s)=(M-1)\,G\,s^{-3}$, then (at least formally) for $\mu_{i,j}=\rho (|\mathbf{x}_i-\mathbf{x}_j|)$
\begin{align}
(L_{\mu}\,\mathbf{x})_i= G\,\sum_{k=1,k\neq i}^n \frac{m_k \left(\mathbf{x}_k-\mathbf{x}_i\right)}{\left|\mathbf{x}_k-\mathbf{x}_i\right|^{3}}\, .
\end{align}
Here we think of $\mathbf{x}:\Gamma\to \RR^3$ as a map that is constant on each cluster.  

We see that the $n$-body problem can be thought of (at least formally) as the nonlinear wave equation 
$
\partial_{tt}\mathbf{x}=L_{\mu}\,\mathbf{x}\, 
$
on the complete graph $\Gamma$.  

Discretize in time yields
\begin{align}
\partial_t\mathbf{x}(\cdot,t+1)-\partial_t\mathbf{x}(\cdot,t)=L_{\mu}\,\mathbf{x}\, .
\end{align}
This can be rewritten as
\begin{align}
\partial_t\mathbf{x}(\cdot,t+1)&=\partial_t\mathbf{x}(\cdot,t)+L_{\mu}\,\mathbf{x}\, ,\\
\mathbf{x}(\cdot,t+1)&=\mathbf{x}(\cdot,t)+\partial_t\mathbf{x}(\cdot,t)\, .
\end{align}
Finally, if we let $A_{\mathbf{x},\mathbf{v}}$ denote the time one map on phase space, that is, on functions $(\mathbf{x},\mathbf{v})$ from $\Gamma$ to $\RR^3\times \RR^3$, then we get from the above that
\begin{align}
A_{\mathbf{x},\mathbf{v}}=\left(\mathbf{x}+\mathbf{v}, \mathbf{v}+L_{\mu}\,\mathbf{x}\right)\, .
\end{align}
For instance, it follows from this and Lemma \ref{l:preserve} that if at the initial time \begin{align}
\sum_km_k\,\mathbf{x}_k=\sum_{k}m_k\,\mathbf{v}_k=0\, , 
\end{align}
then the same holds for $A_{\mathbf{x},\mathbf{v}}$.    Another parallel is:  Even though our equation is nonlinear, it is still the case that if $u$ solves our equation, then so does $u$ plus a constant.  Likewise for the $n$-body problem: If $\mathbf{x}$ solves the $n$-body problem, then so does $\mathbf{x}+c_1\,t+c_0$ for constant vectors $c_0$ and $c_1$. 

For the $n$-body problem the quantities
\begin{align}
I&=\frac{1}{M}\sum_{i<j}m_i\,m_j\,|\mathbf{x}_i-\mathbf{x}_j|^2\, ,\\
U&=\sum_{i<j}m_i\,m_j\,|\mathbf{x}_i-\mathbf{x}_j|^{-1}\, ,
\end{align}
 are respectively called the moment of 
inertia and the potential.   In our language, the moment of inertia is  the energy, whereas the potential is the weighted energy.

\section{Clusters}

We will show next  that clusters of closely aligned members converge more rapidly than the committee as a whole.  We will   assume linear 
decay{\footnote{Some decay is necessary for the conclusions that follow.}}
 for $\rho$; i.e., for $1 \leq s_1 < s_2$
\begin{align}	\label{e:myrho}
	\rho (s_2) \leq    \frac{s_1}{s_2} \,  \rho (s_1)\, .
\end{align}
 
\subsection{Upper bound for the speed}

Since $\rho \leq 1$, we have $ s \, \rho (s) \leq s  \leq 1$ for $0\leq s \leq 1$.  
When $s> 1$, we use \eqr{e:myrho} to get 
\begin{align}
	s \, \rho (s) \leq   \rho (1) \leq 1 \, .
\end{align}
  It follows that the speed is bounded by
\begin{align}
	\|\partial_t\,u (v,t)\|_{\cB} \leq \frac{1}{d}\sum_{\tilde v\approx v} \|u(v,t)-u(\tilde v,t)\|_{\cB} \, \rho( \|u(v,t)-u(\tilde v,t)\|_{\cB} )  \leq
	 \frac{1}{d}\sum_{\tilde v\approx v}  1 = 1 \, .
\end{align}
As an immediate consequence, we get a lower bound for the decay of the gradient:
\begin{align}
	\| \nabla A_u \|_{\infty} \geq \| \nabla  u \|_{\infty} - 2  \, .
\end{align}

\subsection{Clusters}

Given a function $u$ on $\Gamma$, define the $u$-distance 
between   $\Gamma_0 , \Gamma_1 \subset \Gamma$  by
\begin{align}
	\dist_u (\Gamma_0 , \Gamma_1) \equiv \min_{ v \in \Gamma_0 , w \in \Gamma_1 }  \, \|u(v) - u(w)\|_{\cB}   \, .
\end{align}
Given $\Lambda \geq 1$, a  $\Lambda$-cluster is a subset $\Gamma_0 \subset \Gamma$ where
\begin{align}
	\Lambda \, \left\| \nabla \,   u |_{\Gamma_0}    \right\|_{\infty} < \dist_u (\Gamma_0 , \Gamma \setminus \Gamma_0)  \, .
\end{align}
The norm    of the gradient of $u$ restricted to $\Gamma_0$ is given by
\begin{align}
	\left\| \nabla \,   u |_{\Gamma_0}    \right\|_{\infty} = \max_{ v,w \in \Gamma_0 } \,      \|u(v) - u(w)\|_{\cB}  \, .
\end{align}

In addition to \eqr{e:myrho},
we will assume that there is some $C$ so that for $s \geq 1$
\begin{align}
	s\, |\rho'(s)| \leq C \,  \rho (s)  \, .
\end{align}

\begin{Thm}	\label{l:cluster1}
Given   $ \lambda > 1$, there exist $\Lambda_0  , \kappa > 1$ so that if  $  \left| \Gamma_0 \right| > 2$, $\Gamma_0$ is a $\Lambda$-cluster for $u$
 for some $\Lambda \geq \Lambda_0$,   
  $ \lambda \geq \left\| \nabla \,   u |_{\Gamma_0}    \right\|_{\infty}$ and $\dist_u (\Gamma_0 , \Gamma \setminus \Gamma_0)  \geq \Lambda_0$, then 
  \begin{itemize}
  \item $\kappa \, \| \nabla A_u |_{\Gamma_0} \|_{\infty} \leq  \| \nabla u |_{\Gamma_0} \|_{\infty} $.
  \item $\Gamma_0$ is a
   $\kappa \, \Lambda $-cluster for $A_u$.
  \end{itemize}

\end{Thm}

The first conclusion in the theorem says that the cluster is contracting, while the second says that this contraction is   faster  than the rate at which the outliers approach the cluster.
This theorem can be iterated until one of the outliers comes within $\Lambda_0$  of the cluster, with the cluster contracting exponentially all the while.

\vskip2mm
In order to prove Theorem \ref{l:cluster1}, it will be convenient to divide the operator $A$ into two parts. 
We will define the operator $A^0$ on functions on $\Gamma_0$ to be essentially the time one map
 that one would get by ignoring $\Gamma \setminus \Gamma_0$.  Namely, if $v \in \Gamma_0$, then
\begin{align}
	A^0_u (v) = u (v) + \frac{1}{d} \, \sum_{\tv \in \Gamma_0 } (u(\tv) - u(v)) \, \rho (\|u(\tv) - u(v)\|_{\cB} ) \, .
\end{align}
Likewise, we define $\bar{A}_u$ to be the effect of the elements in $\Gamma \setminus \Gamma_0$, so that
\begin{align}
	\bar{A}_u (v) = \frac{1}{d} \, \sum_{w \notin \Gamma_0} (u(w) - u(v)) \, \rho (\|u(w) - u(v)\|_{\cB} ) \, .
\end{align}
It follows that
\begin{align}
	A_u (v) = A^0_u(v) + \bar{A}_u(v) \, .
\end{align}

The next lemma shows that $\bar{A}$ is a contraction on $\Gamma_0$ as long as $\Lambda_0$ is large enough.

\begin{Lem}	\label{l:barA}
Given $v_1 \ne v_2$ in $\Gamma_0$, we have
\begin{align}	\label{e:v1v2}
	\left\| \bar{A}_u (v_1) - \bar{A}_u(v_2) \right\|_{\cB} \leq \|u(v_1) - u(v_2)\|_{\cB} \, 
	\frac{d-d_0}{d}   \,   (C+1) \,  \rho (\dist_u (\Gamma_0 , \Gamma \setminus \Gamma_0)  )  \, .
\end{align}
\end{Lem}

\begin{proof}
We can write $\bar{A}_u (v_1) - \bar{A}_u(v_2)$ as
\begin{align}
	\bar{A}_u (v_1) - \bar{A}_u(v_2) = \frac{1}{d} \, \sum_{w \notin \Gamma_0} B(v_1,v_2,w) \, ,
\end{align}
where we define $B(v_1,v_2,w)$ by
\begin{align}
B(v_1,v_2,w) \equiv
(u(w) - u(v_1)) \, \rho (\|u(w) - u(v_1)\|_{\cB} ) - (u(w) - u(v_2)) \, \rho (\|u(w) - u(v_2)\|_{\cB} ) \, . \notag
\end{align}

Fix some
  $w \notin \Gamma_0$.  After possibly switching $v_1$ and $v_2$, we can assume that
  \begin{align}
  	\dist_u (\Gamma_0 , \Gamma \setminus \Gamma_0) \leq \|u(w) - u(v_2)\|_{\cB} \leq \|u(w) - u(v_1)\|_{\cB} \, .
  \end{align}
We have
\begin{align}
	\left\| B(v_1 , v_2 , w) \right\|_{\cB} &\leq  \|u(v_1) - u(v_2)\|_{\cB} \, \rho (\|u(w) - u(v_1)\|_{\cB} ) \\
	&+ \|u(w) - u(v_2)\|_{\cB} \, \left|
	 \rho (\|u(w) - u(v_1)\|_{\cB} ) -  \rho (\|u(w) - u(v_2)\|_{\cB} )
	\right| \notag \, .
\end{align}
The term on the right in the first line is bounded by $\|u(v_1) - u(v_2)\|_{\cB} \, \rho (\dist_u (\Gamma_0 , \Gamma \setminus \Gamma_0))$ since $\rho$ is monotone.  The second line is bounded by
\begin{align}
	  &\|u(w) - u(v_2)\|_{\cB}  \, \|u(v_1) - u(v_2)\|_{\cB}\,  \sup_{s \geq \|u(w) -    u (v_2)\|_{\cB}} \,   | \rho'(s)|   \\
	  \leq C\,  \|u(v_1) - u(v_2)\|_{\cB} \,  &\rho ( \|u(w) - u(v_2)\|_{\cB})
	   \leq C\,  \|u(v_1) - u(v_2)\|_{\cB} \,  \rho (\dist_u (\Gamma_0 , \Gamma \setminus \Gamma_0))	  
	  \, .   \notag
\end{align}
Adding the two bounds and 
summing  over $w \notin \Gamma_0$ gives the lemma. 
\end{proof}

Next, we will see that $A^0$ decreases the gradient on $\Gamma_0$.

\begin{Lem}	\label{l:A0}
If $\Gamma_0$ has $d_0 +1$ vertices, then we have
\begin{align}	\label{e:fromitst}
	\| \nabla A^0_u \, |_{\Gamma_0}  \|_{\infty, } \leq  
	 \left( 1 - \frac{  \rho (\| \nabla u \, |_{\Gamma_0}\|_{\infty}) \, (d_0-1)}{2d} \right) \, 
	\| \nabla u \, |_{\Gamma_0}  \|_{\infty} 
	\, .
\end{align}

\end{Lem}

\begin{proof}
 The operator $A^0$ is  the time one map for discrete functions on $\Gamma_0$ with the coefficients
 $\tilde{\mu}$ given by
 \begin{align}
 	\tilde{\mu}_{v,w} = \frac{d_0}{d} \, \rho (|u(v) - u(w)|) \geq  \frac{d_0}{d} \, \rho (\| \nabla u \, |_{\Gamma_0}\|_{\infty})  \, ,
 \end{align}
 where the last inequality is the monotonicity of $\rho$.
Theorem
\ref {l:iterateStepV} gives that
\begin{align}	 
	\| \nabla A^0_u \, |_{\Gamma_0}  \|_{\infty } & \leq \left( 1 - \frac{\frac{d_0}{d} \, \rho ( \| \nabla u \, |_{\Gamma_0}\|_{\infty}) \, (d_0-1)}{2d_0} \right) \, 
	\| \nabla u \, |_{\Gamma_0}   \|_{\infty }  \notag \\
	&= 
	 \left( 1 - \frac{  \rho (\| \nabla u \, |_{\Gamma_0}\|_{\infty}) \, (d_0-1)}{2d} \right) \, 
	\| \nabla u \, |_{\Gamma_0}  \|_{\infty } 
	\, .
\end{align}

\end{proof}

\begin{proof}[Proof of Theorem \ref{l:cluster1}]
Define constants $g_0 , g_1 , h_0$ and $h_1$ by
\begin{align}
	 g_0 &= \| \nabla u |_{\Gamma_0} \|_{\infty} \, , && g_1 = \| \nabla A_u |_{\Gamma_0} \|_{\infty} \, , \\
	 h_0 &= \dist_u (\Gamma_0 , \Gamma \setminus \Gamma_0)  , \, 
	 &&h_1  =\dist_{A_u} (\Gamma_0 , \Gamma \setminus \Gamma_0)  \, .
\end{align}
Since $\Gamma_0$ is a $\Lambda$-cluster (for $u$), we have $\Lambda \, g_0 < h_0$.  To prove the theorem, we must show that  $\kappa \, g_1 \leq g_0$ and $ \kappa \, \Lambda \, g_1 < h_1$ for some $\kappa > 1$ as long as $\Lambda$ and $h_0$ are sufficiently large.

To bound $g_1$, 
 choose $v_1 , v_2 \in \Gamma_0$ to maximize $\|A_u (v_1) - A_u (v_2)\|_{\cB}$.   Combining
 \eqr{e:v1v2} (with this choice of $v_1$ and $v_2$) and \eqr{e:fromitst} gives
 \begin{align}	 
	   \frac{ g_1}{g_0}  \leq    1 - \frac{  \rho (g_0) \, (d_0-1)}{2d} +
	\frac{d-d_0}{d}   \,   (C+1) \,  \rho (h_0)  \, .
\end{align}
Because of the decay of $\rho$, we can take $\Lambda_0$ large so that $(C+1) \,  \rho ( h_0) \leq  \frac{  \rho (g_0) \, (d_0-1)}{4d}$ and, thus,
 \begin{align}	 
	   \frac{ g_1}{g_0}  \leq    1 - \frac{  \rho (g_0) \, (d_0-1)}{4d} \leq  1 - \frac{  \rho (\lambda) \, (d_0-1)}{4d}  \, .
\end{align}
This gives $\kappa_1 \, g_1 \leq g_0$ for some $\kappa_1 > 1$.  

Let $\kappa$ be halfway between $\kappa_1$ and one.   We will show that, after possibly taking $\Lambda_0$ larger, we have  $ \kappa \, \Lambda \, g_1 < h_1$.
Since the speed is at most $1$ and $g_0 \geq 1$, the triangle inequality gives 
\begin{align}	\label{e:d1big}
	h_1 \geq h_0 - 2   \, .
\end{align}
It follows that
\begin{align}
	\frac{h_1}{g_1} =  \frac{h_1}{g_0} \, \frac{g_0}{g_1} \geq \kappa_1 \,  \frac{h_0 -2}{g_0} \geq \kappa_1 \, \Lambda \,  \frac{h_0 -2}{h_0} \, .
\end{align}
This gives the desired bound as long as $\Lambda_0 \geq \frac{2\kappa_1}{\kappa_1 - \kappa}$, completing the proof.
 
\end{proof}

Even as clusters are coming together at an exponential rate, all of the members of a cluster may drift off together toward consensus with the rest of the committee.   
 The results for the clusters extend to the case of multiple clusters with obvious modifications, as long as any two are sufficiently far apart.

 \section{Entropies}

We will see that various entropies are monotone  for this evolution equation.  In this section, 
we consider the elementary case of monotonicity under the time one map where
monotonicity will follow from a standard application of convexity and Jensen's inequality.  Later, we will prove sharper estimates for the derivative of the entropy in the continuous time case.

The entropy, \cite{Sh},  of a positive function $u:\Gamma \to \RR$ on a finite graph $\Gamma$ is
\begin{align}
	\Entropy (u) = - \sum_{v \in \Gamma} u(v) \, \log u(v) \, .
\end{align}
Throughout this section, $u$ will be a real-valued function.

\begin{Pro}	\label{p:entropymono}
If $A_u$ is the time one map for $\partial_t - L_{\mu}$ and the $\mu$'s are symmetric, then
\begin{align}
	\Entropy (A_u) \geq \Entropy (u) \, .
\end{align}
If we have equality,  $\Gamma$ has at least three elements and the $\mu$'s are positive, then $u$ is constant.
\end{Pro}

The proposition will be a consequence of the following general lemma:

\begin{Lem}	\label{l:monoF}
Suppose that 
 the function $A$ is given by
\begin{align}
	A (v) =  \sum_{\tv \in \Gamma} a_{v,\tv} \, u(\tv) \, , 
\end{align}
where $a_{v,\tv} \geq 0$,  $\sum_{\tv} a_{v,\tv} = 1$ and $a_{v,\tv} = a_{\tv,v}$.  Given a convex $G: \RR \to \RR$, we have
\begin{align}	\label{e:mF}
	\sum_{v\in \Gamma} G(A(v)) \leq \sum_{v\in \Gamma} G(u(v)) \, .
\end{align}
If $G$ is strictly convex and we have equality in \eqr{e:mF}, then $u$ is constant on $\Gamma_v = \{ \tv \, | \, a_{v,\tv} \ne 0 \}$ for each $v$.
\end{Lem}

\begin{proof}
Since $G$ is convex, Jensen's inequality gives
\begin{align}
	G(A(v)) = G \left( \sum_{\tv} a_{v,\tv} \, u(\tv)  \right) \leq \sum_{\tv} a_{v,\tv} \, G (u(\tv)) \, .
\end{align}
Summing this over $v$ and then switching the order of summation gives
\begin{align}
	\sum_v G(A(v))   \leq \sum_v \sum_{\tv} a_{v,\tv} \, G (u(\tv)) = \sum_{\tv} \left( \sum_v a_{v,\tv} \right) \, G (u(\tv))  \, .
\end{align}
If we have symmetry of the $a_{v,\tv}$'s, then 
\begin{align}
	\sum_v a_{v,\tv} = \sum_v a_{\tv , v} = 1 \, ,
\end{align}
completing the proof of \eqr{e:mF}.

If $G$ is strictly convex and we have equality in \eqr{e:mF}, then for each $v$ 
\begin{align}
	u(\tv_1) = u(\tv_2) {\text{ for every }} \tv_1 , \tv_2 {\text{ with }} a_{v, \tv_1} a_{v,\tv_2} \ne 0 \, .
\end{align}
\end{proof}

If $A$ is the time one map $A_u$ for the operator $\partial_t - L_{\mu}$ with symmetric $\mu$'s,
 then \eqr{e:simplecompu} gives
\begin{align}	 
A_u(v)& = \left(1-\frac{1}{d}\sum_{\tilde v\approx v}\mu_{v,\tilde v}\right)\,u(v)+\frac{1}{d}\sum_{\tilde v\approx v}\mu_{v,\tilde v}\,u(\tilde v)\, .
\end{align}
It follows that $A=A_u$ is of the form required by the lemma with
\begin{align}	\label{e:thea1}
	a_{v,v} &= 1-\frac{1}{d}\sum_{\tilde v\approx v}\mu_{v,\tilde v}  \, , \\
	a_{v,\tv} &= \frac{\mu_{v,\tv}}{d} \,  {\text{ if }} v \ne \tv \, .  \label{e:thea2}
\end{align}
We have that $0 \leq a_{vv} , a_{v,\tv} $ since $0 \leq \mu \leq 1$.

\begin{proof}[Proof of Proposition \ref{p:entropymono}]
The first part of the proposition follows from Lemma \ref{l:monoF} since $ \Entropy (u) =- \sum_v G(u)$ for the (strictly) convex function $G(s) = s \, \log s$. 

 If we have equality for the entropies, then $u$ is constant on the complement of every vertex by Lemma \ref{l:monoF}.  If there are more than two vertices, this forces $u$ to be constant.
\end{proof}

The same argument gives monotonicity of the Renyi entropies, \cite{R}, defined for $\alpha > 0 $, $\alpha \ne 1$, by
\begin{align}
	R_{\alpha} (u) = \frac{1}{1-\alpha} \, \log \sum_{v \in \Gamma} u^{\alpha} (v) \, .
\end{align}

\begin{Pro}	\label{p:renyi}
If $u$ is positive, $A_u$ is the time one map for $\partial_t - L_{\mu}$ and the $\mu$'s are symmetric, then
for $\alpha > 0 $, $\alpha \ne 1$,
\begin{align}
	R_{\alpha} (A_u) \geq R_{\alpha}(u) \, .
\end{align}
\end{Pro}

\begin{proof}
This follows from Lemma \ref{l:monoF} since $\log$ is a monotone function and $F_{\alpha}(s) = s^{\alpha}$ is convex for $\alpha > 1$ and concave for $0< \alpha < 1$.
\end{proof}

In particular, this implies monotonicity of the $L^{\alpha}$ norms:

\begin{Cor}	\label{c:renyi}
If $u$ is positive, $A_u$ is the time one map for $\partial_t - L_{\mu}$ and the $\mu$'s are symmetric, then
for $\alpha > 1 $, 
\begin{align}
	\sum_{v \in \Gamma} A_u^{\alpha} (v) \leq \sum_{v \in \Gamma} u^{\alpha} (v) \, .
\end{align}
\end{Cor}

Taking the limit as $\alpha \to \infty$ gives that the $L^{\infty}$ norms are also monotone.
This monotonicity of the $L^{\alpha}$ norms required that the $\mu$'s are symmetric.  However, the monotonicity of the $L^{\infty}$ norm in Proposition \ref{p:maxprin} did not require symmetry.

\section{Energies and Poincar\'e inequalities}	\label{s:poinc}

We will next define a discrete energy that is natural for the operator $L_{\mu}$ and prove a weighted Poincar\'e inequality for this energy. 
Throughout this section, the vector-valued map $u$ on $\Gamma$ will go into a Hilbert space $\cH$  with norm $| \cdot |$ and inner product $\langle \cdot , \cdot \rangle$.

Given a function $\sigma : \RR \to \RR$, define the $\sigma$-weighted energy by
\begin{align}
	\|\nabla u \|_{2,\sigma}^2 \equiv  \frac{1}{d} \,  \sum_{(\tv , v) \in \cE}  |u(\tv) - u(v)|^2 \, \sigma (|u(\tv) - u(v)|) \, ,
\end{align}
where $\cE$ denotes the set of unoriented edges defined in \eqr{e:edges}.
 To keep notation short, we set $\|\nabla u \|_{2}=  \|\nabla u \|_{2,1}$ when 
$\sigma \equiv 1$.    The natural energy associated to the operator $L_{\mu}$ with $\mu_{v,\tv} = \rho (|u(v) - u(\tv)|)$ is  when  $\sigma = \rho$.
 
  \begin{Lem}	\label{l:energyr}
 If $u$ is a function on $\Gamma$ and $L_{\mu}$ has $\mu_{v,\tv} = \rho (|u(v) - u(\tv)|)$, then 
 \begin{align}	\label{e:egyb}
 	\| \nabla u \|^2_{2,\rho} = - \sum_v \, \langle u(v) ,  L_{\mu} u (v) \rangle = - \sum_v \, \langle u(v)- \cA_u ,
		L_{\mu} u (v) \rangle \, .
 \end{align}
 Therefore, Cauchy-Schwarz gives
  \begin{align}	\label{e:egyb2}
 	\| \nabla u \|^2_{2,\rho} \leq  \| u- \cA_u \|_2 \,  \| L_{\mu} u \|_2 \, .
 \end{align}
 \end{Lem}
 
 \begin{proof}
 We have
  \begin{align}
 	\sum_v \, \langle u(v) \, , L_{\mu} u (v) \rangle &= \frac{1}{d} \, \sum_{\tv \ne v} \langle u(v)\, , (u(\tv) - u(v)) 
	\rangle \, \rho (|u(\tv) - u(v)|) \notag \\
	&= \frac{1}{d} \, \sum_{\tv \ne v} \langle u(\tv)\,  , (u(v) - u(\tv)) \rangle \, \rho (|u(\tv) - u(v)|) \, , 
 \end{align}
 where the second equality is interchanging $v$ and $\tv$.  Adding the two equations gives
  \begin{align}
 	2\, \sum_v \, \langle u(v) \, , L_{\mu} u (v) \rangle &= - \frac{1}{d} \, \sum_{\tv \ne v}  |u(\tv) - u(v)|^2 \, \rho (|u(\tv) - u(v)|) \notag \\
		&= - \frac{2}{d} \, \sum_{(\tv , v) \in \cE}  |u(\tv) - u(v)|^2 \, \rho (|u(\tv) - u(v)|)  \, , 
 \end{align}
 giving the first equality in \eqr{e:egyb}.  The second equality in \eqr{e:egyb} uses  $\sum_v L_{\mu} u(v) = 0$.
 \end{proof}

 \begin{Lem}	\label{l:Lene}
 If $u$ is a function on $\Gamma$ and $L_{\mu}$ has $\mu_{v,\tv} = \rho (|u(v) - u(\tv)|)$, then 
 \begin{align}
 	\| L_{\mu} u \|_2^2 \leq 2 \, \| \nabla u \|_{2,\rho^2}^2 \, .
 \end{align}
 \end{Lem}
 
 \begin{proof}
 Cauchy-Schwarz gives
 \begin{align}
 	\| L_{\mu} u \|_2^2 &= \sum_v \left| \frac{1}{d} \, \sum_{\tv \approx v} ( u(\tv) - u(v) ) \, \rho (|u(\tv) - u(v)|) \right|^2 \notag \\
		&\leq  \sum_v  \frac{1}{d} \, \sum_{\tv \approx v} | u(\tv) - u(v) |^2 \, \rho^2 (|u(\tv) - u(v)|) = 
		2 \, \| \nabla u \|_{2,\rho^2}^2 \, .
 \end{align}
 
 \end{proof}

  \subsection{Weighted Poincar\'e inequalities}
  
   The next proposition is a nonlinear Poincar\'e inequality for the $\rho$-energy  or, equivalently, a nonlinear eigenvalue estimate for $L_{\mu}$, when 
   $\rho (s) \geq s^{-\alpha}$.  The special case $\alpha =0$ and $\rho$ is a constant
    is the standard linear Poincar\'e inequality.
 
 \begin{Pro}		\label{l:nonpoin}
 Suppose that $u: \Gamma \to \RR$ and $\rho (s) \geq s^{-\alpha}$ for some $\alpha \geq 0$.
 If $0 < \alpha \leq 1$, then 
 \begin{align}
 	\| u - \cA_u \|_2^{2-\alpha} 
		 \leq  2 \, \left( \frac{d}{d+1}  \right)^{2-\alpha} \,  \| \nabla u \|_{2,\rho}^2 \, .   \label{e:pnoi}
 \end{align}
  When $1< \alpha <2$, we have
   \begin{align}
 	\| u - \cA_u \|_2^{2-\alpha} 
		 \leq      \frac{2d}{\left(d+1 \right)^{2-\alpha}} \,  \| \nabla u \|_{2,\rho}^2 \, .
 \end{align}
 \end{Pro}

\begin{proof}
It suffices to prove the case   $\rho (s) = s^{-\alpha}$.
When $0\leq \alpha\leq 1$, we have $2-\alpha \geq 1$.  For each $v$, 
  the H\"older inequality (or the triangle inequality when $\alpha = 1$) gives
\begin{align}
 	\left|     \frac{1}{d} \, \sum_{\tv \ne v} (u (\tv) - u(v))    \right| 
		 &\leq \left(    \frac{1}{d}  \sum_{\tv \ne v} |u(\tv) - u(v)|^{2-\alpha} 
		 \right)^{ \frac{1}{2-\alpha} } 
		 \, .
 \end{align}
Multiplying by $\frac{d}{d+1}$, squaring, and summing over $v$ gives
\begin{align}
 	\| u - \cA_u \|_2^2 &= \sum_v \left(    \frac{1}{d+1} \, \sum_{\tv \ne v} (u (\tv) - u(v))    \right)^2
		 \leq  \frac{d^2}{(d+1)^2} \, \sum_v \left(    \frac{1}{d}  \sum_{\tv \ne v} |u(\tv) - u(v)|^{2-\alpha} 
		 \right)^{ \frac{2}{2-\alpha} } 
		 \notag \\
		  &\leq  \frac{d^2}{(d+1)^2} \, \left( \frac{1}{d} \,  \sum_v    \sum_{\tv \ne v} |u(\tv ) - u(v)|^{2-\alpha} 
		 \right)^{ \frac{2}{2-\alpha} } =  \frac{d^2}{(d+1)^2}  \left( 2\,  \| \nabla u \|^2_{2,\rho} \right)^{ \frac{2}{2-\alpha} } 
		 \, ,
 \end{align}
 where the last inequality used Lemma \ref{l:elemi} below.
 
 Suppose next that $1< \alpha < 2$.  Using \eqr{e:xipP} below with $p = \frac{1}{2-\alpha} > 1$  
 gives for each $v$
 \begin{align}
 	     \sum_{\tv \ne v} |u (\tv) - u(v)|
		 &\leq \left(    \sum_{\tv \ne v} |u(\tv) - u(v)|^{2-\alpha} 
		 \right)^{ \frac{1}{2-\alpha} } 
		 \, .
 \end{align}
 Squaring this, summing over $v$ and applying \eqr{e:xip} below gives
 \begin{align}
 	\| u - \cA_u \|_2^2 &=  \frac{1}{(d+1)^2}  \sum_v \left(   \, \sum_{\tv \ne v} (u (\tv) - u(v))    \right)^2
		 \leq  \frac{1}{(d+1)^2} \, \sum_v \left(     \sum_{\tv \ne v} |u(\tv) - u(v)|^{2-\alpha} 
		 \right)^{ \frac{2}{2-\alpha} } 
		 \notag \\
		  &\leq  \frac{1}{(d+1)^2} \, \left(    \sum_v    \sum_{\tv \ne v} |u(\tv ) - u(v)|^{2-\alpha} 
		 \right)^{ \frac{2}{2-\alpha} }  = \frac{1}{(d+1)^2}  \left( 2\, d\, \| \nabla u \|^2_{2,\rho} \right)^{ \frac{2}{2-\alpha} } 
		 \, .
 \end{align}
 
 \end{proof}

 In some of the main cases of interest, we restrict to functions $\rho$ with $\rho \leq 1$.  The next corollary gives a Poincar\'e inequality that can be applied in these cases where we do not have $\rho\geq s^{-\alpha}$ for all $s>0$.  The function $\rho$ is always assumed to be nonincreasing.
 
 \begin{Cor}		\label{c:poin1}
  Suppose that $u: \Gamma \to \RR$ and $\rho (s) \geq c\, s^{-\alpha}$ for $s \geq 1$ for some $c  > 0$ and $\alpha \in (0,2)$.  There exists $C_{c,d,\alpha}$ depending on $c$, $d$ and $\alpha$ so that
 \begin{align}
 	\| u - \cA_u \|_2^{2-\alpha} 
		 \leq  C_{c,d,\alpha} \,   \left( 
		   \| \nabla u \|_{2,\rho}^2 +   \| \nabla u \|_{2,\rho}^{2-\alpha}
		 \right) \, .   
 \end{align}
   \end{Cor}

 \begin{proof}
If  $s \geq 1$, then $s^{2-\alpha} \leq \frac{s^2 \rho (s)}{c}$.  When $s< 1$, then $c \leq \rho (s)$ and 
 $s^{2-\alpha} \leq \left( \frac{s^2 \rho (s)}{c} \right)^{ \frac{2-\alpha}{2} }$.  It follows that we get for all $s$ that
\begin{align}
	s^{2-\alpha} \leq  \frac{s^2 \rho (s)}{c} + \left( \frac{s^2 \rho (s)}{c} \right)^{ \frac{2-\alpha}{2} } \, .
\end{align}
Summing over the vertices and applying Proposition \ref{l:nonpoin} gives the corollary.
  \end{proof}

 \begin{Lem}	\label{l:elemi}
 If $p>1$ and $x_i \geq 0$ for $i=1 , \dots , n$, then
 \begin{align}	\label{e:xip}
 	\sum_i x_i^p \leq \left( \sum_i x_i \right)^p \, .
 \end{align}
 Equivalently, we get
 \begin{align}	\label{e:xipP}
 	\sum_i x_i \leq \left( \sum_i x_i^{ \frac{1}{p} } \right)^p \, .
 \end{align}
 \end{Lem}
 
 \begin{proof}
 Since \eqr{e:xip} implies \eqr{e:xipP}, it suffices to prove \eqr{e:xip}.
 We may assume that $\sum_i x_i  > 0$ since the lemma holds trivially otherwise.  
 Define $y_i \geq 0$ by setting 
 \begin{align}
 	y_i \equiv \frac{ x_i}{ \sum_i x_i } \, .
 \end{align}
 It follows that $\sum y_i = 1$ and $y_i^p \leq y_i$ and, thus,
 \begin{align}
 	 \frac{ \sum_i x_i^p }{ \left( \sum_i x_i \right)^p} = \sum \left(  \frac{ x_i}{ \sum_i x_i }  \right)^p =  \sum_i  y_i^p \leq \sum_i y_i = 1 \, .
 \end{align}
 \end{proof}

\section{The continuous case}

Let $u:\Gamma \times [0,\infty) \to \cH$, so that time is now continuous, and satisfying $\partial_t u = L_{\mu} u$
where $\mu_{v,\tv} = \rho (|u(\tv,t) - u(v,t)|)$.  The time continuous case models  continuous  committee deliberations.
The next lemma computes the evolution of quantities $G(u)$ depending on $u$.  For generality, we will consider maps
$G$ from $\cH$ to a second Hilbert space $\cH_0$ that are Frechet differentiable with derivative $dG$.
 
\begin{Lem}	\label{l:ucvxG}
If $u$ satisfies $\partial_t u = L_{\mu} u$ and $G:\cH \to \cH_0$ is differentiable, 
then{\footnote{The sum on the right hand side of \eqr{e:8p2} is over all $v$ and $\tilde v$ with $v\ne \tilde v$; we use this convention for the sum except when $v$ is already fixed.}}
\begin{align}	\label{e:8p2}
	  \partial_t \, \sum_v G(u(v,t)) &=  - \frac{1}{2d} \,
	 \sum_{\tv \ne v} \left( dG_{u(\tv,t)} - dG_{u(v,t)} \right) \, (u(\tv,t) - u(v,t)) \, \rho (|u(\tv,t) - u(v,t)|)
   \, .
\end{align}
\end{Lem}

\begin{proof}
Since $u_t(v) = \frac{1}{d} \, \sum_{\tv \ne v} (u(\tv) - u(v)) \, \rho  (|u(\tv,t) - u(v,t)|)$, 
 we have
\begin{align}
	 \partial_t \, \sum_v G(u(v,t))  &=   \frac{1}{d} \, \sum_{\tv \ne v} dG_{u(v,t)} \, (u(\tv,t) - u(v,t)) \, \rho (|u(\tv,t) - u(v,t)|) \, .
\end{align}
Adding and subtracting $dG_{u(\tv,t)} \, (u(\tv,t) - u(v,t)) \, \rho (|u(\tv,t) - u(v,t)|)$ gives
\begin{align}
	  \partial_t \, \sum_v G(u(v))  &= -  \frac{1}{d} \,  \sum_{\tv \ne v} (dG_{u(\tv,t)} - dG_{u(v,t))}\, (u(\tv,t) - u(v,t)) \, \rho (|u(\tv,t) - u(v,t)|) \notag \\
	&\qquad +   \frac{1}{d} \,  \sum_{\tv \ne v} dG_{u(\tv,t)} \, (u(\tv,t) - u(v,t)) \, \rho (|u(\tv,t) - u(v,t)|) \, .
\end{align}
The   last line is minus $   \partial_t \, \sum_v G(u(v))$, so the lemma follows by adding this to each side.
\end{proof}

As a consequence, we see that the average of $u$ is preserved:

\begin{Cor}	\label{c:usquared}
If $u$ satisfies $\partial_t u = L_{\mu} u$, then $  \partial_t \, \cA_{u (\cdot , t)} = 0$ and 
\begin{align}
	  \partial_t \,  \| u - \cA_u \|_{2}^2 &= \partial_t \, \| u \|_2^2 = - 2\,  \|\nabla u(\cdot , t) \|_{2,\rho}^2  \, .
\end{align}
\end{Cor}

\begin{proof}
The first claim follows from  Lemma \ref{l:ucvxG} with $G(s) = s$, so that $dG_s$ is the identity.

The first equality in the second claim follows from the first claim.
Applying Lemma \ref{l:ucvxG} with $G(s) = |s|^2$, so that $dG_s (\cdot) = 2\langle s , \cdot \rangle$, gives
\begin{align}
	   \partial_t \, \| u \|_2^2 &=  \partial_t \, \sum_v |u(v,t)|^2  = -  \frac{1}{d} \,  \sum_{\tv \ne v}  |u(\tv,t) - u(v,t)|^2 \, \rho (|u(\tv,t) - u(v,t)|) \, ,
\end{align}
giving the second equality and completing the proof.
\end{proof}

\subsection{Entropies in the continuous case}

We now specialize to the case where $u$ is real-valued and $G: \RR \to \RR$.

\begin{Cor}	\label{c:gmini}
If $u$ satisfies $\partial_t u = L_{\mu} u$ and $G:\RR \to \RR$, then
\begin{align}
	  \partial_t \, \sum_v G(u(v,t)) &\leq  - \frac{1}{2d} \,
	 \sum_{\tv \ne v}  (u(\tv,t) - u(v,t))^2 \, \rho (|u(\tv,t) - u(v,t)|) \, \min_s G''(s)
   \, ,
\end{align}
where  the minimum is over $s$  between $ u(\tv,t)$ and $u(v,t) $.  
In particular, we have
\begin{align}
	  \partial_t \, \sum_v G(u(v,t)) &\leq  -  \min G''  \,
	\| \nabla u \|_{2,\rho}^2
   \, .
\end{align}
\end{Cor}

\begin{proof}
Given $\tv, v, t$, the mean value theorem gives
\begin{align}
	G'(u(\tv,t)) - G'(u(v,t)) = (u(\tv,t) - u(v,t))\, G'' (s) \, , 
\end{align}
where $s$ is some value between $u(\tv,t)$ and $u(v,t)$.  It follows that
\begin{align}
	(G'(u(\tv,t)) - G'(u(v,t)))\, (u(\tv,t) - u(v,t))&= (u(\tv,t) - u(v,t))^2 \,   G''(s)\notag \\
	 &\geq (u(\tv,t) - u(v,t))^2 \, \min_{\tau} G''(\tau) \, ,
\end{align}
where the minimum is taken over $\min \{ u(\tv,t) , u(v,t)\}  \leq \tau \leq \max \{ u(\tv,t) , u(v,t)\} $.  
  Lemma \ref{l:ucvxG} gives
\begin{align}
	  \partial_t \, \sum_v G(u(v,t)) &\leq  - \frac{1}{2d} \,
	 \sum_{\tv \ne v}  (u(\tv,t) - u(v,t))^2 \, \rho (|u(\tv,t) - u(v,t)|) \, \min_s G''(s)
   \, ,
\end{align}
giving the corollary.
\end{proof}

Similarly, we get monotonicity of the entropy $\Entropy (t) \equiv \Entropy (u(\cdot , t))$:

\begin{Cor}
If $u>0$ satisfies $\partial_t u = L_{\mu} u$, then
\begin{align}	\label{e:entrml}
	 \Entropy'(t) &=    \frac{1}{2d} \, \sum_{\tv \ne v}  \log \frac{ u(\tv,t)}{u(v,t)} \,  (u(\tv,t) - u(v,t)) \, \rho (|u(\tv,t) - u(v,t)|) \, .
\end{align}
\end{Cor}

\begin{proof}
This follows by applying Lemma \ref{l:ucvxG} with $G(s) = -s \log s$, so that $G'(s) = -1 - \log s$ and noting that $\sum_v u_t (v) = 0$.
\end{proof}

\begin{Cor}
If $u>0$ satisfies $\partial_t u = L_{\mu} u$, then
\begin{align}	\label{e:entrmlc}
\Entropy '(t) \geq \frac{1}{2d} \, \sum_{\tv \ne v} 
\frac{( u(\tv,t) - u(v,t))^2 \, \rho (|u(\tv,t) - u(v,t)|)} 
{ \max \{ u(\tv,t) , u(v,t)\} }
\, .
\end{align}
\end{Cor}

\begin{proof}
Applying Corollary \ref{c:gmini} with $G(s) = s \log s$, so that 
$G''(s) = \frac{1}{s}$, we have
\begin{align}
\partial_t \, \sum_v (u(v,t))\, \log u(v,t) &\leq - \frac{1}{2d} \,
\sum_{\tv \ne v} \, \frac{(u(\tv,t) - u(v,t))^2 \, \rho (|u(\tv,t) 
- u(v,t)|)}{\max \{ u(\tv,t) , u(v,t)\}}
\, .
\end{align}
Multiplying by $-1$ gives the corollary.
\end{proof}

Similarly, we get monotonicity for the Renyi entropies $R_{\alpha}$ 
(cf. Proposition \ref{p:renyi}).

\begin{Cor}
If $u>0$ satisfies $\partial_t u = L_{\mu} u$ and $\alpha > 0, \alpha \ne 1$, then
\begin{align}	\label{e:ryi}
\partial_t &\, \exp{\{ (1-\alpha)R_{\alpha} (t)\}}\notag\\
=&
- \frac{\alpha}{2d} \,
\sum_{\tv \ne v} ((u(\tv,t))^{\alpha-1} - (u(v,t))^{\alpha -1})\, 
(u(\tv,t) - u(v,t)) \, \rho (|u(\tv,t) - u(v,t)|)
\, .
\end{align}
\end{Cor}

\begin{proof}
If we set $G_{\alpha} (s) = s^{\alpha}$, then $G_{\alpha}'(s) = 
\alpha \, s^{\alpha -1}$ and
\begin{align}
\exp{\{ (1-\alpha)R_{\alpha} (t)\}} = \sum_{v \in \Gamma} G_{\alpha} 
(u(v,t)) \, .
\end{align}
The corollary now follows from
Lemma \ref{l:ucvxG}.
\end{proof}

When $\alpha > 1$, $s^{\alpha -1}$ is nondecreasing and   \eqr{e:ryi}  is nonpositive; the opposite holds when $\alpha < 1$.  The quantity $R_{\alpha}(t)$ is nondecreasing in either case.

  \section{Decay of the variance}	\label{s:variance}
  
 For a map $u:\Gamma\to \cH$, the variance is given by
\begin{align}
\Var_u=\frac{1}{d+1}\, \sum_{v}|u(v)-\cA_u|^2 = \frac{1}{d+1} \,  \|u\|^2_2  - \left| \cA_u \right|^2 \, .
\end{align}
 We will assume that $u: \Gamma \times \RR \to \RR$ satisfies $\partial_t u = L_{\mu} u$ with $\mu = \rho$.  
 Corollary \ref{c:usquared}
gives
\begin{align}
	   \partial_t \,  \| u - \cA_u \|_{2}^2 &= \partial_t \, \| u \|_2^2 = - 2\,  \|\nabla u(\cdot , t) \|_{2,\rho}^2    \, .
\end{align}

 The next theorem proves decay for the variance when $\rho$ is bounded from below by a negative power.  This power condition is natural in many cases; cf. Section \ref{s:nbody} where $\rho$ is exactly a negative power in the $n$-body problem.   The power condition is not natural in the problem of committee rankings where $\rho$ is always assumed to be at most one; the second theorem below will deal with this case.
 As the power goes to zero in the next theorem, we recover the exponential decay of the variance for the ordinary heat equation.

\begin{Thm}	\label{p:vardecay}
If $\rho (s) \geq s^{-\alpha}$ for   $0\leq \alpha < 2$, then as long as $u$ is not constant
	\begin{align}
	 \left(  \Var_{u(\cdot , t)}^{  \frac{\alpha}{2} }  \right)_t \leq - \, 
	 \begin{cases}
	   \frac{ \alpha  \, (d+1)^{ \frac{4-3\alpha}{2} } }{2d^{2-\alpha} }    {\text{ if }} 0< \alpha \leq 1  \, , \\
	    \frac{ \alpha  \, (d+1)^{ \frac{4-3\alpha}{2} } }{2d }  {\text{ if }} 1< \alpha < 2  \, .\\
	   \end{cases}
\end{align}
In the last case $\alpha =0$,  we have $\left( \log \Var_{u(\cdot , t)} \right)_t \leq -    \frac{(d+1)^2}{d^2} $.
\end{Thm}

\begin{proof}
Set $f(t) = (d+1) \, \Var_{u(\cdot , t)}$ and $g(t) = \| \nabla u(\cdot , t) \|_{2,\rho}^2 $.
If $0 \leq \alpha \leq 1$, then 
Proposition \ref{l:nonpoin} gives
 \begin{align}
 	f^{\frac{2-\alpha}{2}} 
		 \leq  2\, \left( \frac{d}{d+1}  \right)^{2-\alpha} \,  g =  \left( \frac{d}{d+1}  \right)^{2-\alpha} \,  (-f')
		 \, .   \label{e:pnoig}
 \end{align}
 The case $\alpha =0$ follows immediately.
Similarly, when    $1< \alpha <2$,
 Proposition \ref{l:nonpoin} gives
 \begin{align}
 	f^{\frac{2-\alpha}{2}} 
		 \leq    \left( \frac{d}{(d+1)^{2-\alpha}}  \right) \,  (-f')
		 \, .    
 \end{align}
   When $\alpha \ne 0$, the theorem follows from this since
 \begin{align}
 	\frac{2}{\alpha} \, \left( f^{\frac{\alpha}{2}} \right)' = \frac{f'}{ f^{ \frac{2-\alpha}{2} } } \, .
 \end{align}

\end{proof}

As a consequence, we see that $u$ must become constant in finite time when $0 < \alpha < 2$.

\begin{Cor}	\label{c:vardecay}
If $\rho (s) \geq s^{-\alpha}$ for   $0< \alpha < 2$ and  $u(\cdot ,t)$ is not constant for some $t>0$, then
\begin{align}
	\Var_{ u (\cdot , t)}^{\frac{\alpha}{2}}  \leq \Var_{ u (\cdot , 0)}^{\frac{\alpha}{2}} - c_{d,\alpha} \, t \, ,
\end{align}
where 
\begin{align}
c_{d , \alpha} = 
 \begin{cases}
	   \frac{ \alpha  \, (d+1)^{ \frac{4-3\alpha}{2} } }{2d^{2-\alpha} }    {\text{ if }} 0< \alpha \leq 1  \, , \\
	    \frac{ \alpha  \, (d+1)^{ \frac{4-3\alpha}{2} } }{2d }  {\text{ if }} 1< \alpha < 2  \, .\\
	   \end{cases}
	   \end{align}
\end{Cor}

\begin{proof}
This follows by integrating Theorem \ref{p:vardecay}.
\end{proof}

Corollary \ref{c:vardecay} bounds the time to consensus depending on $d$, $\alpha$ and the initial variance.  When $d$ is large but fixed and there are a few outliers, this bound is stronger than the one from the exponential decay of $\| \nabla u \|_{\infty}$.

We turn next to the case where $\rho (s) \geq c \, s^{-\alpha}$ for $s \geq 1$, where $\alpha , c > 0$.   The next theorem shows that the variance goes to zero in this case as well.  The variance to the power $\frac{\alpha}{2}$ decreases a definite amount until the variance gets below one; from then on, it decays exponentially.

\begin{Thm}	\label{p:vardecay2}
If $\rho (s) \geq c\, s^{-\alpha}$ for   $s\geq 1$, where $0\leq \alpha < 2$, then  
	\begin{itemize}
	 \item
	 $ \left(  \Var_{u(\cdot , t)}^{  \frac{\alpha}{2} }  \right)_t \leq -
	c_{c,d,\alpha}   {\text{ if }}   1 \leq \Var_{u(\cdot , t)}$.
	 \item  $\left( \log \Var_{u(\cdot , t)} \right)_t \leq
	 -
	c_{c,d,\alpha}   {\text{ if }}   \Var_{u(\cdot , t)} < 1$.
	\end{itemize}
Here the constant $c_{c,d,\alpha} $ depends on $c$, $d$ and $\alpha$.
\end{Thm}

\begin{proof}
Set $f(t) = (d+1) \, \Var_{u(\cdot , t)}$.
Corollary \ref{c:poin1} gives a differential inequality of the type 
\begin{align}
	f^{ \frac{2-\alpha}{2} } \leq C \, \left(|f_t| + |f_t|^{ \frac{ 2- \alpha}{2} } \right) \, ,
\end{align}
where $f_t \leq 0$ and $C$ depends on $c$, $d$ and $\alpha$.  For $f \geq 1$, this leads to a bound
\begin{align}
	f^{ \frac{2-\alpha}{2} } \leq - C' \, f_t \, ,
\end{align}
while for $f\leq 1$ it gives
\begin{align}
	f \leq - \bar{C} \, f_t \, .
\end{align}
The theorem follows from this.
\end{proof}

\section{Monotonicity of weighted energy}	\label{s:monoE}

Let $u: \Gamma \times \RR \to \cH$ satisfy $\partial_t u = L_{\mu} u$.
Given a function $\rho (s)$, define a function $\sigma(s)$ by
\begin{align}	\label{e:sigma2}
	 \sigma (s) = \frac{ 2 \, \int_0^s\tau \, \rho (\tau)\,d\tau }{s^2} \, .
\end{align}
It follows that $\left( s^2 \, \sigma \right)' = 2 \, s \, \rho$ and, thus, that for $s> 0$
\begin{align}	\label{e:siga}
	\sigma (s) + \frac{s}{2} \sigma'(s) = \rho (s)  \, .
\end{align}

\begin{Pro}	\label{p:sig}
If   $\sigma$ satisfies \eqr{e:siga}, then
\begin{align}
	 \partial_t \, \|\nabla u (\cdot ,t) \|_{2,\sigma}^2 & =  - 2\, \| u_t (\cdot , t) \|_2^2   	  \, .
\end{align}
\end{Pro}

\begin{proof}
The derivative of the $\sigma$-weighted energy is given by
\begin{align}
	d \, \partial_t \, \|\nabla u (\cdot , t) \|_{2,\sigma}^2 &=
	   \sum_{\tv \ne v}  \langle (u_t(\tv,t) - u_t(v,t)) \, , (u(\tv,t) - u(v,t)) \rangle \, \sigma (|u(\tv,t) - u(v,t)|)  \notag \\
	   &\qquad +  \frac{1}{2}\, \sum_{\tv \ne v}   |u(\tv,t) - u(v,t)|^2 \, \partial_t \sigma (|u(\tv,t) - u(v,t)|) \, .
\end{align}
 Set $f = u (\tv , t) - u(v , t)$ to simplify notation.
 The chain rule gives
  \begin{align}
	  \partial_t \sigma (|f|) = \sigma'(|f|) \langle \frac{f}{|f|} , \, f_t \rangle \, .
\end{align}
Substituting this in, using \eqr{e:siga}, and then using that $f$ is skew in $v$ and $\tv$  gives
 \begin{align}
	d \, \partial_t \, \|\nabla u (\cdot , t) \|_{2,\sigma}^2 &=
	   \sum_{\tv \ne v}  \left(\langle  f_t  ,  \,f \, \rangle \sigma (|f|)    + 
	    \frac{1}{2}\,    |f|^2 \, \sigma'(|f|) \langle \frac{f}{|f|} \, , f_t \rangle  \right) 
	    =  \sum_{\tv \ne v}   \langle f_t  ,  \,f \, \rangle \rho (|f|)
	     \notag \\
	    &=  -2 \, \sum_{\tv \ne v}    \langle u_t  (v,t)  \, , f   \rangle	  \, \rho (|f|)  
	    =  -2 \, d\, \sum_{ v}    |u_t|^2 (v,t) \, .
\end{align}
   
\end{proof}

 \subsection{When $\rho = s^{-\alpha}$}

We consider next    where   $\rho (s) = s^{-\alpha}$ is homogeneous.    We will consider three cases, depending on $\alpha$.  
 First, if $0< \alpha < 2$, then   
\begin{align}	\label{e:intalp}
	\sigma(s) = 2\, s^{-2} \, \int_0^s \tau^{1-\alpha} \, d\tau =  \frac{2}{2-\alpha} \,  s^{-\alpha} = 
      \frac{2}{2-\alpha} \, \rho (s) \, .
\end{align}
Proposition \ref{p:sig} gives that
  $\| \nabla u(\cdot , t) \|_{2,\rho}$ is nonincreasing and
\begin{align}	\label{e:fromfirst}
	  \partial_t \, \|\nabla u (\cdot ,t) \|_{2,\rho}^2 & =  (\alpha -2) \, \| u_t (\cdot , t) \|_2^2  	   	\, .  
\end{align}
If $\alpha > 2$, then the integral in \eqr{e:intalp} diverges; however, 
$\sigma(s)   = 
      \frac{2}{2-\alpha} \, \rho (s)$ satisfies \eqr{e:siga} and \eqr{e:fromfirst} still holds.  The right hand side is   nonpositive, so $\|\nabla u (\cdot ,t) \|_{2,\rho}^2$ is nondecreasing.{\footnote{In the special case
where  $\rho (s) = s^{-3}$, the weighted (or potential) energy is increasing; cf. Section \ref{s:nbody}.}}
Finally, when $\alpha =2$, we get a formal solution $\sigma (s) = 2s^{-2} \, \log s$ and 
Proposition \ref{p:sig} gives 
\begin{align}
	\partial_t \, \| \nabla u (\cdot ,t) \|_{2,s^{-2} \, \log s}^2 = -  \| u_t (\cdot , t) \|_{2}^2 \, .
\end{align}

 \subsection{Properties of $\sigma$ in general}
 
The next lemma collects a few useful properties of $\sigma$ in general.

\begin{Lem}
When $\rho:[0,\infty)\to [0,1]$ is monotone nonincreasing and $\sigma$ given by \eqr{e:sigma2}, then for all $0<r<s$
\begin{align}
\rho (s)&\leq \sigma (s)\leq \sigma (0)=\rho (0)\, ,\label{e:above}\\
\sigma'(s)&\leq 0\, ,\label{e:noninc}\\
\sigma (s)&\leq \left(\frac{r}{s}\right)^2\,\sigma (r)+\left(1-\left(\frac{r}{s}\right)^2\right)\, \rho (r)\leq  \left(\frac{r}{s}\right)^2\,\sigma (0)+\left(1-\left(\frac{r}{s}\right)^2\right)\, \rho (r)\, .\label{e:convsorts}
\end{align}
Inequality \eqr{e:convsorts} implies that if $\rho$ goes to zero at infinity, then so does $\sigma$.  

\end{Lem}

\begin{proof}
It follows immediately from the definition of $\sigma$ that $\sigma (0)=\rho (0)$.  
Since $\rho$ is nonincreasing we have that 
\begin{align}
	s^2 (\sigma (s) - \rho(s) ) &= 2 \, \int_0^s \tau \, \rho (\tau) \, d \tau - 2 \, \int_0^s \tau \, \rho (s) \, d\tau
	= 2\, \int_0^s \tau (\rho (\tau) - \rho (s)) \, d\tau \geq 0 \, . 
\end{align}
Combining this with \eqr{e:siga} gives
\begin{align}
	\frac{s}{2} \, \sigma'(s) = \rho(s) - \sigma (s) \leq 0 \, . 
\end{align}
It follows from the above that \eqr{e:above} and \eqr{e:noninc} hold.   Finally, since
\begin{align}
s^2\,\sigma (s)&=2\int_0^r \tau\,\rho (\tau)\,d\tau+2\int_r^s\tau\,\rho (\tau)\,d\tau\notag\\
&= r^2\,\sigma(r)+2\int_r^s\tau\,\rho (\tau)\,d\tau
\leq r^2\,\sigma (r)+(s^2-r^2)\,\rho (r)\, ,
\end{align} 
we get \eqr{e:convsorts}.  
\end{proof}



 \section{A nonlinear three circles theorem for the variance}	\label{s:circles}

We will assume that $u: \Gamma \times \RR \to \cH$ satisfies $\partial_t u = L_{\mu} u$ with $\mu = \rho$ and $\rho (s) = s^{-\alpha}$ for some  $\alpha \in (0,2)$.  
We will set 
\begin{align}
I(t)&= \left( \Var_{u(\cdot , t)} \right)^{ \frac{\alpha}{2} } \, ,\\
U(t) & = I'(t) \,  .
\end{align}
The $U$ defined this way will be thought of as a nonlinear frequency.  
These are the same quantities that came up in Theorem \ref{p:vardecay} which 
gives a negative upper bound for $U$.

The main result of this section is the convexity of $I$ (or equivalently monotonicity of $U$).  This convexity  does not depend on $d$ or on the dimension of the target.

\begin{Thm}  \label{t:threecircles}
$I$ is nonnegative, nonincreasing, and convex.   Consequently, the variance itself is nonincreasing and convex.
\end{Thm}

This has the following nonlinear Hadamard's three circles type theorem as consequence (the same conclusions hold for the variance):

\begin{Cor}	\label{c:threecircles}
For $0<r<s$
\begin{align}
	I(r) &\leq \frac{s-r}{s}\,I(0)+\frac{r}{s}\,I(s)\, ,\\
     I(0)-I(s)&\leq \frac{s}{r}\,\left(I(0)-I(r)\right) \, . \label{e:125}
\end{align}
\end{Cor}

\begin{proof}
The first claim follows from convexity of $I$ and the second follows from the first.    
\end{proof}


\vskip2mm
Theorem \ref{t:threecircles} is an immediate consequence of the next lemma:

\begin{Lem}	\label{l:diffine}
If $0 \leq \alpha < 2$,  then for $ f(t)=(d+1) \,  \Var_u = \|u-\cA_u\|^2_2$ and 
$g(t)=\|\nabla u\|^2_{2,\rho}$
\begin{align}
 	g_t=\partial_t \, \| \nabla u \|_{2,\rho}^2  & \leq - (2-\alpha) \, \frac{ \| \nabla u \|_{2,\rho}^4 }{ (d+1) \, \Var_u } =- (2-\alpha) \, \frac{g^2}{f}   \, , \label{e:122} \\
	 \left( f^{\frac{ \alpha}{2}}   \right)_{t t} &=  -\alpha \,\left( \frac{ g}{ f^{ \frac{2-\alpha}{2} } } \right)_t \geq 0 \, . \label{e:1113} 
\end{align}
\end{Lem}

\begin{proof}
Corollary \ref{c:usquared} gives that 
\begin{align}
   f_t &= \partial_t \|u-\cA_u\|_2^2
   =- 2\, \|\nabla (u-\cA_u)\|_{2,\rho}^2 = - 2\, g \, .
 \end{align}
The derivative of $g$ is given by \eqr{e:fromfirst}
 \begin{align}
   g_t &=\partial_t \|\nabla (u-\cA_u)\|_{2,\rho}^2= (\alpha -2) \,\|(u-\cA_u)_t\|_2^2
   		 = (\alpha -2) \,\| L_{\mu} u \|_2^2\, .
  \end{align}
  Since $g^2 \leq f \, \| L_{\mu} u \|_2^2$ by Lemma  \ref{l:energyr}, we
get
  \eqr{e:122}.  
  
The first equality in \eqr{e:1113} follows from the chain rule since $f_t = -2g$. We then use that
\begin{align}
	\left( g \, f^{ \frac{\alpha -2}{2} } \right)_t &= 
	 g \, f^{ \frac{\alpha -2}{2} }  \, \left\{  \frac{g_t }{g}   + \left( \frac{\alpha -2}{2}  \right)\,
		  \frac{f_t}{f}    \right\} =   g \, f^{ \frac{\alpha -2}{2} }  \, \left\{  \frac{g_t }{g}   + \left( 2- \alpha   \right)\,
		  \frac{g}{f}    \right\}  \leq 0 \, ,
\end{align}
where the last inequality used    \eqr{e:122}.  
    
\end{proof}

\section{Variance in the discrete case}

We now return to the discrete case where $u:\Gamma  \to \cH$ and $A_u$ is the time one map for $\partial_t - L_{\mu}$ with $\mu = \rho$.   

\begin{Lem}	\label{vardiff}
We have
\begin{align}
	(d+1) \, \left( \Var_{A_u} - \Var_u \right)=   -2 \, \| \nabla u \|_{2,\rho}^2 + \| L_{\mu} u \|_2^2 \,  .
\end{align}
\end{Lem}

\begin{proof}
Since $\cA_{A_u} = \cA_u$ and the variance is unchanged when we subtract a constant, we  have
\begin{align}
	(d+1) \, \left( \Var_{A_u} - \Var_u \right)=  \| A_u \|_2^2 - \| u \|_2^2 \, .
\end{align}
At each $v$ in $\Gamma$, we have
\begin{align}
	|A_u|^2 - |u|^2  = 2\langle u\, , (A_u - u) \rangle + |A_u - u|^2 
	= 2\, \langle u \, , L_{\mu} u \rangle  + \left| L_{\mu} u \right|^2 \, .
\end{align}
The lemma follows by summing this over $v$  and 
applying Lemma \ref{l:energyr}.
 \end{proof}

\begin{Cor}	\label{c:deltavar}
We have
\begin{align}
	(d+1) \, \left( \Var_{A_u} - \Var_u \right) 
		\leq 2 \, \left( \max \rho - 1 \right) \, \| \nabla u \|_{2,\rho}^2 \,  .
\end{align}
\end{Cor}

\begin{proof}
Combining Lemmas \ref{vardiff} and  \ref{l:Lene} gives
\begin{align}
	(d+1) \, \left( \Var_{A_u} - \Var_u \right)=   -2 \, \| \nabla u \|_{2,\rho}^2 + \| L_{\mu} u \|_2^2 
		\leq 2 \, \left( \max \rho - 1 \right) \, \| \nabla u \|_{2,\rho}^2 \,  .
\end{align}
\end{proof}

\begin{Thm}	\label{p:vardecay3}
If $\rho (s) \geq c\, s^{-\alpha}$ for   $s\geq 1$ where $0< \alpha < 2$ and $\rho \leq \rho_0 < 1$, then  
	\begin{itemize}
	 \item
	 $ \Var_{A_u}^{  \frac{\alpha}{2} } -   \Var_{u}^{  \frac{\alpha}{2} } \leq -
	C  {\text{ if }}   1 \leq \Var_{u}$.
	 \item  $\log \Var_{A_u} - \log \Var_u  \leq -C
	   {\text{ if }}   \Var_{u} < 1$.
	\end{itemize}
Here the constant $C > 0$ depends on $c$, $d$, $\alpha$ and $\rho_0$.
\end{Thm}

\begin{proof}
Corollary \ref{c:deltavar}
gives $C_1 = C_1 (d,\rho_0) > 0$ so that 
\begin{align}	\label{e:C1}
	  \Var_{A_u} - \Var_u  
		\leq -C_1 \,  \| \nabla u \|_{2,\rho}^2 \,  .
\end{align}
Corollary \ref{c:poin1}
gives $C_2 = C_2(c,d,\alpha)$ so that 
 \begin{align}	\label{e:C2here}
 	\Var_u^{ \frac{2-\alpha}{2} } 
		 \leq  C_2  \, \left(  \| \nabla u \|_{2,\rho}^2 +  \| \nabla u \|_{2,\rho}^{2-\alpha} 
		 \right) \, .   
 \end{align}
Suppose that $\Var_u \leq 1$.  Lemma \ref{l:eleme} below (with $y= \Var_u$, $x=\| \nabla u \|^2_{2,\rho}$ and $p=\frac{2-\alpha}{2}$)
gives $C_2'>0$ so that
$\Var_u 
		 \leq  C_2'  \,    \| \nabla u \|_{2,\rho}^2$ and, thus, \eqr{e:C1} gives
 \begin{align}	\label{e:C1a}
	  \Var_{A_u} - \Var_u  
		\leq -C_1 \,  \| \nabla u \|_{2,\rho}^2 \leq - \frac{C_1 }{ C_2' } \, \Var_u  \,  ,
\end{align}
giving the decay in this case.

If $\Var_u \geq 1$, then Lemma \ref{l:eleme} gives $	\Var_u^{ \frac{2-\alpha}{2} } 
		 \leq  C_2'  \,    \| \nabla u \|_{2,\rho}^2$ and, thus,
\begin{align}
	\Var_{A_u} \leq \Var_u - \frac{C_1 }{ C_2' } \, \Var_u^{ 1 - \frac{\alpha}{2}} \, .
\end{align}
 The decay in this case now follows from Lemma \ref{l:moredecay} below (with $x = \Var_{A_u}$, $y= \Var_u$ and $p= \frac{\alpha}{2}$).
\end{proof}

\begin{Lem}	\label{l:eleme}
Suppose that $0 < p < 1$, $C > 0$, and $x,y > 0$ satisfy
	$y^p \leq C \, (x + x^p) \, .$
There exists $C' = C' (p,C)$ so that
\begin{itemize}
\item $y \leq C' \, x$ if $y\leq 1$.
\item  $y^p \leq C' \, x$ if $1 \leq y$.
\end{itemize}
\end{Lem}

\begin{proof}
Suppose first that $y \leq 1$.  If $1\leq x$,  then  the first claim holds (with $C' =1$).  On the other hand, if $x \leq 1$, then we have
\begin{align}
	\frac{y^p}{x^p} \leq C \, \left( x^{1-p} + 1 \right) \leq 2\, C \, .
\end{align}
If $1\leq y \leq C \, (x + x^p)$, then $x$ is bounded away from zero and, thus, $x^{p-1}$ is bounded.  Since
\begin{align}
	\frac{y^p}{x} \leq C \, \left( 1 + x^{p-1} \right) \, , 
\end{align}
this gives the second claim.  
\end{proof}

\begin{Lem}	\label{l:moredecay}
Given $C > 0$ and $p \in (0,1)$, there exists $C'$ so that if $0< x,y$, $1 \leq y$, and
$
	x-y \leq - C \, y^{1-p} \, , 
$
then $x^p - y^p \leq - C'$.
\end{Lem}

\begin{proof}
Since  $x \leq y - C \, y^{1-p}$, it suffices to get a negative upper bound for the function
\begin{align}
	G(y) \equiv \left( y - C \, y^{1-p} \right)^p - y^p
\end{align}
for all $y \geq 1$.  Obviously $G(y) < 0$ for all $y$ since $C > 0$, so it suffices to prove that $G$ cannot go to zero as $y \to \infty$.  We will do this by showing that $G' \leq 0$ for $y$ large enough.  We have
\begin{align}
	\frac{G'(y)}{p\, y^{p-1}}  =    \frac{ 1 - C \, (1-p) \, y^{-p} }{\left( 1 - C \, y^{-p} \right)^{1-p}}  - 1 \, .
\end{align}
 To see that $G'(y) \leq 0$ for all large $y$, we use that for all small $\eta \ne 0$ we have
\begin{align}
	(1-\eta)^{1-p} >  1 - (1-p) \, \eta \, , 
\end{align}
as can be seen by Taylor expanding $(1-\eta)^{1-p} $ about $\eta = 0$.

\end{proof}

 \section{More general models}  	\label{s:genmod}
 
\subsection{Other models}

Another model one may consider is where how much change an individual committee member is willing to make depends on his/her ranking of that candidate.  For instance, one may consider a case where an individual committee  member is much less likely to make big changes in her/his ordering if she/he ranks
 a candidate near the top of the list as opposed to near the bottom of the list.  
 In this case 
\begin{align}
\mu_{v,\tilde v}=\rho \left(\| u(v) \|_{\cB} ,\|u(\tilde v)-u(v)\|_{\cB}\right)\, ,
\end{align}
where $\rho: [0,\infty) \times [0,\infty)\to (0,1]$ and $\rho$ is  monotone nonincreasing  both variables.   That is, if $\rho=\rho (r,s)$ and $s_0$ is fixed, then $\rho(r,s_0)$ is monotone nonincreasing  in $r$; and if $r_0$ is fixed, then $\rho (r_0,s)$ is monotone nonincreasing in $s$.  This gives rise to a fully quasilinear discrete heat equation
$
\partial_t\,u=L_{\mu}\,u\, ,
$
where
\begin{align}
\mu^t_{v,\tilde v}=\rho (\|u(v,t) \|_{\cB} , \|u(v,t)-u(\tilde v,t)\|_{\cB})\, .
\end{align}
This discrete differential equation is the graph version of a fully quasilinear equation on $\RR^n\times\RR$ given by
$
\partial_t\,u=\sum_{i,j}a_{i,j}(u,\nabla u)\, u_{i,j}\, .
$

Note that in this more general case $\mu$ depends on the orientation of an edge.   This results in that the overall opinion of the committee of a candidate may not be constant in time and in general the views of the committee do  not converge to the average.    However, even in this case, we still get exponential convergence to consensus:

\begin{Thm}   \label{t:main2}
If $\partial_t\,u=L_{\mu}\,u$, $\mu^t_{v,\tilde v}=\rho ( \|u(v,t) \|_{\cB} , \|u(v,t)-u(\tilde v,t)\|_{\cB} )$, and 
\begin{align}
a= \rho (\max \| u(\cdot,0) \|_{\cB} ,\|\nabla u(\cdot,0)\|_{\infty})>0\, , 
\end{align}
then 
\begin{align}
	\| \nabla u(\cdot,t)  \|_{\infty} &\leq \left( 1 - \frac{a(d-1)}{2d} \right)^t \, 
	\| \nabla u (\cdot,0)  \|_{\infty} \, .
\end{align}
\end{Thm}

\begin{proof}
This follows from Theorem \ref{t:main1} since by \eqr{e:maxdown} and Theorem \ref{l:iterateStepV}
\begin{align}
\max_{v} \|u(v,t) \|_{\cB} &\leq \max_{v} \|u(v,0) \|_{\cB} \, ,\\
\|\nabla u(\cdot,t)\|_{\infty}&\leq \|\nabla u(\cdot,0)\|_{\infty}\, .
\end{align} 
\end{proof}

\subsection{Other weights}

Another model is where
$
\partial_t\,u=L_{\mu}\,u\, ,
$
and
\begin{align}
\mu_{v,\tilde v}=d\,\frac{\rho^2 (\|u(v)-u(\tilde v)\|_{\cB})}{\sum_{w\approx v}\rho (\|u(v-u(w)\|_{\cB})}\, .
\end{align}
This   arises when one argues that the new $u(v,t+1)$ should be a weighted sum where the different vertices should not any more have the same weight but it should depend on $\rho (\|u(v)-u(\tilde v)\|_{\cB})$.  Thus, the uniform weight $\frac{1}{d}$ is replaced by 
\begin{align}
	\frac{\rho (\|u(v)-u(\tilde v)\|_{\cB})}{\sum_{w\approx v}\rho (\|u(v-u(w)\|_{\cB})}
\end{align}
 which still sums to one, but is no longer uniform.  This gives an even higher weight to closer opinions.  The coefficients $\mu_{v,\tv}$ would then depend on more than just $\|u(v) - u(\tv)\|_{\cB}$, but the argument still extends to cover this case with obvious modifications.


\begin{thebibliography}{Abc}

\bibitem[C]{C}
F.  R. K. Chung,  
{\emph{Spectral graph theory}}.  
CBMS Regional Conference Series in Mathematics, 92; American Mathematical Society, Providence, RI, 1997.

\bibitem[G]{G} 
A. Grigoryan, {\it{Analysis on Graphs}}, preprint.

\bibitem[KKS]{KKS}
S. Kannan, S. Khanna, and M. Sudan. {\emph{Personal Communication}}, May 2015.

\bibitem[Mi]{Mi}J.  Milnor, \emph{The geometry of the Kepler problem}, AMS Notices 90 (June-July 1983), 353--365.

\bibitem[Mo1]{Mo1} J.  Moser, \emph{Is the Solar System Stable?} The Mathematical Intelligencer, 1 (1978), 65--71.

\bibitem[Mo2]{Mo2} \bysame, \emph{Dynamical systems - past and present}. Proceedings of the International Congress of Mathematicians, Vol. I (Berlin, 1998). Doc. Math. 1998, Extra Vol. I, 381--402.  

\bibitem[R]{R}
A. R\'enyi, \emph{On measures of information and entropy}, Proceedings of the fourth Berkeley Symposium on Mathematics, Statistics and Probability (1960) 547--561.

\bibitem[Sh]{Sh}
C. E. 
Shannon, 
{\emph{A mathematical theory of communication}}. 
Bell System Tech. J. 27, (1948). 379--423, 623--656.

\bibitem[S]{S}
D. Spielman,  {\emph{Algorithms, graph theory, and linear equations in Laplacian matrices}}. Proc. of the International Congress of Math. Volume IV, 2698--2722, Hindustan Book Agency, New Delhi, 2010.

\bibitem[Su]{Su}
T.
Sunada,  
 {\emph{Discrete geometric analysis.  Analysis on graphs and its applications}}, 51--83, 
Proc. Sympos. Pure Math., 77, Amer. Math. Soc., Providence, RI, 2008. 

\end{thebibliography}
\end{document}